\documentclass[11pt]{amsart}
\usepackage{amsmath,amsthm, amscd, amssymb,amsfonts}
\usepackage[all]{xy}
\usepackage{mathrsfs}
\usepackage{enumitem}
\usepackage[T2A,T1]{fontenc}
\usepackage{bigstrut}
\usepackage{hyperref}

\newcommand{\sh}{\mbox{\usefont{T2A}{\rmdefault}{m}{n}\cyrsh}}

\newcommand{\sch}{\mbox{\usefont{T2A}{\rmdefault}{m}{n}\cyrshch}}

\usepackage{caption}

\usepackage[ansinew]{inputenc}
\usepackage{graphicx,fancyhdr}

\usepackage[dvips, dvipsnames, usenames]{color}

\newcommand{\comment}[1]{}
\newcommand{\af}{\mathsf{a}}

\numberwithin{equation}{section}
\theoremstyle{plain}
\newtheorem{theorem}{Theorem}[section]
\newtheorem*{theorem-main}{Theorem}
\newtheorem{lemma}[theorem]{Lemma}

\newtheorem{prop}[theorem]{Proposition}

\theoremstyle{definition}

\newtheorem{example}[theorem]{Example}

\theoremstyle{remark}

\def\pf{\begin{proof}}
\def\epf{\end{proof}}

\newcommand{\fkdos}{\hspace{-1pt}k+\frac{1}{2}}
\newcommand{\fudos}{\hspace{-1pt}\frac{3}{2}}

\newcommand{\ba}{ \mathbf{a}}

\newcommand{\bm}{ \mathbf{m}}
\newcommand{\bn}{ \mathbf{n}}
\newcommand{\ku}{ \Bbbk}
\newcommand{\kut}{ \ku^{\times}}

\newcommand{\G}{\mathbb G}

\newcommand{\ghost}{\mathscr{G}}

\newcommand{\I}{\mathbb I}
\newcommand{\Iw}{\mathbb I^{\dagger}}
\newcommand{\Idd}{\mathbb I^{\ddagger}}

\newcommand{\N}{\mathbb N}

\newcommand{\bq}{\mathbf{q}}

\newcommand{\Z}{\mathbb Z}
\newcommand{\zt}{\Z^{\theta}}

\renewcommand{\_}[1]{_{\left( #1 \right)}}

\newcommand{\Ss}{{\mathcal S}}

\newcommand{\cV}{\mathcal{V}}

\newcommand{\lstr}{\mathfrak L}

\newcommand{\pos}{\mathfrak P}
\newcommand{\eny}{\mathfrak E}

\newcommand\ad{\operatorname{ad}}

\newcommand{\car}{\operatorname{char}}
\newcommand{\Der}{\operatorname{Der}}

\newcommand{\id}{\operatorname{id}}

\newcommand{\Hom}{\operatorname{Hom}}
\newcommand{\ord}{\operatorname{ord}}
\newcommand{\lcm}{\operatorname{lcm}}

\def\ydh{{}^{H}_{H}\mathcal{YD}}

\newcommand{\NA}{\mathscr{B}}
\newcommand{\toba}{\mathscr{B}}
\newcommand{\ot}{\otimes}

\newcommand{\ydG}{{}^{\ku \Gamma }_{\ku \Gamma }\mathcal{YD}}
\newcommand{\ydU}{{}^{\ku \Upsilon }_{\ku \Upsilon }\mathcal{YD}}

\allowdisplaybreaks

\newcounter{tabla}\stepcounter{tabla}

\begin{document}

\title[Pointed Hopf algebras in characteristic $2$]{Examples of finite-dimensional pointed Hopf algebras in characteristic $2$}

\author[Andruskiewitsch, Bagio, Della Flora, Fl\^ores]
{Nicol\'as Andruskiewitsch, Dirceu Bagio, Saradia Della Flora, Daiana Fl\^ores}

\address{FaMAF-Universidad Nacional de C\'ordoba, CIEM (CONICET),
\newline Medina A\-llen\-de s/n, Ciudad Universitaria, 
\newline (5000) C\' ordoba, Rep\'ublica Argentina.} \email{andrus@famaf.unc.edu.ar}

\address{Departamento de Matem\'atica, Universidade Federal de Santa Maria, \newline 
97105-900, Santa Maria, RS, Brazil} \email{bagio@smail.ufsm.br, saradia.flora@ufsm.br, flores@ufsm.br}

\thanks{This material is based upon work supported by the National Science Foundation under
Grant No. DMS-1440140 while N. A. was in residence at the Mathematical Sciences
Research Institute in Berkeley, California, in the Spring 2020 semester.
The work of N. A. was partially supported by CONICET,
Secyt (UNC) and  the Alexander von Humboldt Foundation
through the Research Group Linkage Programme}

\begin{abstract}
We present new examples of finite-dimensional Nichols algebras over fields of characteristic 2 from braided vector spaces that are not of diagonal type, admit realizations as Yetter-Drinfeld modules over finite abelian groups and are analogous to  Nichols algebras of finite Gelfand-Kirillov dimension in characteristic 0.
New finite-dimensional pointed Hopf algebras over fields of characteristic 2 are obtained by bosonization with group algebras of suitable finite abelian groups.
\end{abstract}

\maketitle

\section{Introduction}
The goal of this paper is to present new examples of 
finite-dimensional Hopf algebras in characteristic 2,
which are pointed, non-commutative and non-cocommutative.
Following the usual guidelines of the lifting method, we focus on finite-dimensional Nichols algebras,
then the Hopf algebras are obtained routinely by bosonization.
The main result of \cite{aah-triang} (in characteristic 0)
is the classification of the Nichols algebras with finite Gelfand-Kirillov dimension arising from braided vector spaces $(V,c)$ that decompose as
\begin{align*}
V &=  V_{1} \oplus \dots \oplus V_t \oplus  V_{t + 1} \oplus \dots \oplus V_\theta,
&
c(V_i \otimes V_j) &=  V_j \otimes V_i,\, i,j\in \I_{\theta},
\end{align*}
where  $V_{1}, \dots, V_t$  are blocks (see \S \ \ref{subsection:yd});  
$V_{t+1}, \dots, V_{\theta}$ are points (i.e. have dimension 1);
and the braidings have a specific form, see e. g. \eqref{eq:braiding-block-point},  \eqref{eq:braiding-several-blocks-1pt}. 
This result relies on the classification in \cite{h-classif} and  assumes a Conjecture treated partially in \cite{aah-diag}, both about Nichols algebras of diagonal type. 
However in positive characteristic the classification of finite-dimensional Nichols algebras of diagonal type
is known only in rank $\leq 4$ \cite{heck-wang,wang-rank3,wang-rank4}. 
Inspired by \cite{clw} and by familiar phenomena in Lie theory in positive characteristic, 
examples of finite-dimensional Nichols algebras in odd characteristic were constructed in \cite{aah-oddchar} by
analogy with the Nichols algebras in \cite{aah-triang}--that have infinite dimension.
Here we extend these constructions assuming that the base field $\ku$ is  
algebraically closed of characteristic $2$. There are new features as $1 = -1$ now. 
For instance in characteristic 0, two main actors are the Jordan 
and the super Jordan planes. Their restricted versions in characteristic $p > 2$ have dimensions 
$p^2$ \cite{clw} and $4p^2$ \cite{aah-oddchar} respectively. When $\car \ku  =2$  they merge in the 
restricted Jordan plane that has dimension $16 = 4\times 2^2$ \cite{clw}. Other families of \cite{aah-triang}
also merge. Finally the fact that $x_i^2 = 0$ for suitable $x_i$ in the braided vector space
brings on more examples with finite dimension.
Let us present the main result of this paper.

\begin{theorem-main}
If $V$ is a braided vector space as in Table \ref{tab:main-result}, then the dimension of the Nichols algebra
$\NA(V)$ is finite.
\end{theorem-main}

\renewcommand{\arraystretch}{1.4}
\begin{table}[ht]
\caption{{\small Finite-dimensional Nichols algebra in characteristic 2}}\label{tab:main-result}
\begin{center}
\begin{tabular}{|c|c|c|c|}\hline 
$V$ & {\small $\NA(V)$}&  {\small $\dim K$} & {\small $\dim \NA(V)$}  \\ \hline 
$\lstr_{\wp}(1,1)$ & {\rm Proposition} {\rm\ref{nich-ex1}}  & $2^3$  & $2^7$ \\\hline  
$\lstr_{\wp}(1,a),\,a\neq 1$ & {\rm Proposition} {\rm \ref{nich-ex2}}  & $2^4$ & $2^8$ \\\hline 
$\,\,\pos(\bq,\ba)$, $\ba \in (\kut)^t$ &{\rm Proposition} {\rm\ref{generatores-relations-blocks-point}} & $2^{|\mathcal{A}|}$   &$2^{4t+|\mathcal{A}|}$ \\\hline
$\eny_{\wp}(1)$ & {\rm Proposition} {\rm\ref{nich-pale-1}} &$2^2$  & $2^4$ \\[.2em] \hline
$\eny_{\wp}(\omega),\,\omega\in\G'_3$ & {\rm Proposition} {\rm \ref{nich-pale-2}}   & $3^3$& $2^23^3$ \\[.2em] \hline
\end{tabular}
\end{center}
\end{table}

See \ref{subsubsection:splitting} for the meaning of $K$.
The braided vector spaces $\lstr_{\wp}(1,1)$ appear to be close to $\lstr(-1, \ghost)$ and $\lstr_{-1}(-1, \ghost)$
in \cite[Table 1]{aah-oddchar}, but
$\toba(\lstr_{\wp}(1,a))$, $a\neq 1$ has no finite-dimensional analogue in $\car \ku = p >2$.
Similarly, the algebras $\toba(\pos(\bq,\ba))$ are finite-dimensional in odd characteristic  
only when the entries of $\ba$ belong to the prime field, in contrast with characteristic 2.
Also $\eny_{\wp}(\omega)$ does not appear in the \emph{loc. cit.}
Albeit no classification is envisageable yet as the knowledge of diagonal type is still incomplete,
we present partial results in Theorems \ref{thm:point-block},  \ref{thm:points-block} and \ref{teo-paleblock}.

After spelling out some preliminaries in Section \ref{section:Preliminaries}, we devote Sections 
\ref{section:block-point}, \ref{section:block-points}, \ref{section:blocks-point} and \ref{sec:paleblock}
to Nichols algebras of one block and one point, one block and several points, several  blocks and one point,
and one pale block and one point respectively. Our proofs rely on the splitting technique 
\S\ref{subsubsection:splitting} and the classifications in \cite{heck-wang,wang-rank3,wang-rank4}.
Explicit examples of finite-dimensional pointed Hopf 
algebras are discussed in \S \ref{subsec:realization-block-pt}, \S \ref{subsec:realizations-blocks-pt},
\S \ref{subsec:realizations-paleblock-pt}. More examples by lifting will be presented in a future work.

\section{Preliminaries}\label{section:Preliminaries}

\subsection{Notations and Conventions}

We denote the natural numbers by $\N$, $\N_0=\N\cup \{0\}$. We set $\I_{k,\ell}=\{n \in \N_0: k \leq n \leq \ell\}$, $\I_{\ell}=\I_{1, \ell}$ and $\N_{\ge \ell}=\N \setminus\I_{\ell-1}$, for $k<\ell \in\N_0$.  We work over an algebraically closed field $\ku$ of characteristic $2$.
The group of $N$-th roots of unity in $\ku$ is denoted by $\G_N$; $\G'_N$ is the subset of the primitive roots of order $N$ and $\G_{\infty}=\bigcup_{N\in \N}\G_N$. 

Throughout $H$ is a Hopf algebra with bijective antipode $\Ss$. We use the notations
$G(H) =$ the group of grouplikes in $H$, $\mathcal{P}(H) =$ the space of primitive elements,
$\widehat{H} = \Hom_{\text{alg}}(H, \ku)$,
$\ydh =$ the category of Yetter-Drinfeld modules over $H$; 
see e.~g. \cite[11.6]{radford-book}.

\subsection{Yetter-Drinfeld modules}\label{subsection:yd} 
\subsubsection{Braided vector spaces}
A braided vector space $V$ is a pair $(V, c)$ where $V$ is a vector space and $c \in GL(V^{\ot 2})$ is a solution of the braid equation
\begin{align*}
(c \ot \id) (\id \ot c) (c \ot \id) &= (\id \ot c) (c \ot \id)(\id \ot c).
\end{align*}
We are interested in two classes of braided vector spaces. First,
$(V, c)$ or simply $V$ is of diagonal type if there exist a basis
$(x_i)_{i\in \I_\theta }$ of $V$ and a matrix $\bq = (q_{ij})_{i,j\in \I_\theta}$
such that $q_{ij} \in \kut$ and $c(x_i\otimes x_j)=q_{ij}x_j\otimes x_i$ for all $i,j\in \I_\theta $.
We denote in $T(V)$, or any quotient braided Hopf algebra,
\begin{align*}
x_{ij} &= (\ad_c x_i)\, x_j,& x_{i_{1}i_2 \dots i_M} &= (\ad_c x_{i_{1}})\, x_{i_2 \dots i_M},&
i,j,i_{1}, \dots, i_M &\in \I, & M&\ge 2.
\end{align*}
Second, let  $\epsilon\in \kut$ and $\ell \in \N_{\ge 2}$.
A \emph{block}  $\cV(\epsilon,\ell)$ is a braided vector space
with a basis $(x_i)_{i\in\I_\ell}$ such that for $i, j \in \I_\ell$, $j>1$:
\begin{align}\label{equation:basis-block}
c(x_i \ot  x_{1}) &= \epsilon x_{1} \ot  x_i,& c(x_i \ot  x_j) &=(\epsilon x_j+x_{j-1}) \ot  x_i.
\end{align}
For simplicity a block $\cV(\epsilon,2)$ of dimension $2$ is called an $\epsilon$-block.

\subsubsection{Realizations}\label{subsubsection:yd-ab}
Any Yetter-Drinfeld module $V$ bears a structure of brai\-ded vector space by $c(v \otimes w) = v\_{-1} \cdot w \otimes v\_{0}$,
$v,w\in V$ where $\delta(v) = v\_{-1} \otimes v\_{0}$.
The braided vector spaces  above appear as Yetter-Drinfeld modules 
in different ways called \emph{realizations}.
Let $\Gamma$ be an abelian group and let $\widehat \Gamma$ be the group of  characters  of $\Gamma$.
The Yetter-Drinfeld modules over the group algebra $\ku \Gamma$ 
are the $\Gamma$-graded $\Gamma$-modules,
the $\Gamma$-grading being denoted by $V = \oplus_{g\in \Gamma} V_g$; thus $h \cdot V_g=V_g$ for $g,h\in \Gamma$. 
If $g\in \Gamma$ and $\chi \in \widehat\Gamma$, then  the one-dimensional vector space $\ku_g^{\chi}$,
with action and coaction given by $g$ and $\chi$, is in  $\ydG$.
Given  $V \in \ydG$ with  a basis $(v_i)_{i\in I}$ where $v_i$ is homogeneous of degree $g_i$,  there are skew derivations $\partial_i$, $i\in I$, of $T(V)$ such that
\begin{align}\label{eq:skewderivations}
\partial _i(v_j) & =\delta _{ij}, \ i, j \in I,&
\partial _i(xy) &= \partial _i(x)(g_i\cdot y)+x\partial _i(y), \ x,y\in T(V).
\end{align}
More generally a \emph{YD-pair} for $H$  is a pair
$(g, \chi) \in G(H) \times \widehat{H}$ such that
\begin{align}\label{eq:yd-pair}
\chi(h)\,g  &= \chi(h\_{2}) h\_{1}\, g\, \Ss(h\_{3}),& h&\in H.
\end{align}
Let $\ku_g^{\chi}$ be a one-dimensional vector space with $H$-action and $H$-coaction  given by $\chi$ and $g$
respectively; then \eqref{eq:yd-pair} says that  $\ku_g^{\chi} \in \ydh$. 
Thus a realization of $V$ of diagonal type with matrix $\bq = (q_{ij})_{i,j\in \I_\theta}$
is just a collection $(g_{1}, \chi_{1}), \dots, (g_{\theta}, \chi_{\theta})$ such that
$q_{ij} =\chi_j (g_i)$ for all $i, j \in \I_{\theta}$.

\subsubsection{Realizations of $\epsilon$-blocks}\label{subsec:realizations-block}
For $\chi\in \widehat{H}$,  the space of $(\chi, \chi)$-deri\-vations is
\begin{align*}
\Der_{\chi,\chi}(H, \ku) &= \{\eta\in H^*: \eta(h\ell) = \chi(h)\eta(\ell) + \chi(\ell)\eta(h) \, \forall h,\ell \in H\}.
\end{align*}

The realizations of $\epsilon$-blocks are given by the notion of \emph{YD-triple} for $H$ \cite{aah-oddchar}; this 
is a collection $(g, \chi, \eta)$ where $(g, \chi)$, is a YD-pair for $H$, $\eta \in \Der_{\chi,\chi}(H, \ku)$, $\chi(g) = \epsilon$, $\eta(g) = 1$ and
\begin{align}\label{eq:YD-triple}
\eta(h) g &= \eta(h\_2) h\_{1} g \Ss(h\_3), & h&\in H.
\end{align}
Given a YD-triple $(g, \chi, \eta)$ we define $\cV_g(\chi,\eta) \in \ydh$ as the
vector space with a basis $(x_i)_{i\in\I_2}$, whose $H$-action and $H$-coaction are given by
\begin{align*}
h\cdot x_{1} &= \chi(h) x_{1},& h\cdot x_2&=\chi(h) x_2 + \eta(h)x_{1},& \delta(x_i) &= g\otimes x_i,&
h&\in H, \, i\in \I_2.
\end{align*}
Then $\cV_g(\chi, \eta)\simeq \cV(\epsilon,2)$ as braided vector spaces. 

\begin{example}
Let $\epsilon = 1$ and $\Gamma = \langle g \rangle$ be a cyclic group of order $N$. Let
$\cV$ be the vector space with a basis $(x_i)_{i\in\I_2}$ with grading $\deg x_i = g$, $i\in \I_2$.
Then the assignment $g \longmapsto \begin{pmatrix} 1 & 1 \\ 0 & 1 \end{pmatrix}$ defines a representation of $\Gamma$
(hence a structure of Yetter-Drinfeld module over $\ku \Gamma$) if and only if  $N$ is even. Thus if $\dim H < \infty$ and $H$ admits a YD-triple (for $\epsilon = 1$), then $\dim H$ is even.
\end{example}

\subsection{Nichols algebras}\label{subsection:nichols algebras} 
 Let $V \in \ydh$. The Nichols algebra of $V$ is the unique
graded connected Hopf algebra $\toba(V) = \oplus_{n\ge 0} \toba^n(V)$ in $\ydh$  
such that $V \simeq \toba^1(V) = \mathcal P(\toba(V))$ generates $\toba(V)$
as algebra. See \cite{a-leyva} for an exposition.The algebra and coalgebra underlying $\toba(V)$ depend only on the braiding.
If $V\in \ydG $ is as in \S~\ref{subsection:yd}, then
the $\partial_i$'s induce skew-derivations on $\toba(V)$.
Then $w\in \toba^k(V)$, $k\ge 1$, is $0$ if and only if
$\partial_i(w)=0$ in $\toba(V)$ for all $i\in I$.

\subsubsection{The restricted Jordan plane}\label{subsubsection:Jordan-plane}
This is the Nichols algebra of a $1$-block.

\begin{theorem}\label{thm:jordan-plane-char2} \cite{clw}
 The algebra $\toba(\cV(1,2))$ is presented by generators $x_{1},x_2$ and relations
\begin{align} \label{eq-jordan-plane-char2}
&x_{1}^2,&  &x_2^4,& &x_2^2x_{1}+x_{1}x_2^2+x_{1}x_2x_{1},& &x_{1}x_2x_{1}x_2+x_2x_{1}x_2x_{1}.&
\end{align}
Let $x_{21}:=x_{1}x_2+x_2x_{1}$. Then $\dim \toba(\cV(1,2))=16$ since $\toba(\cV)$ has a basis 
\begin{align*}
\{x_{1}^{m_{1}}x_{21}^{m_2}x_2^n\,:\,m_{1},m_2\in\I_{0,1},\,n\in \I_{0,3}\}. \qed
\end{align*}
\end{theorem}

\subsubsection{The splitting technique}\label{subsubsection:splitting}
Let $V = V_{1} \oplus V_2$ be a direct sum of objects in $\ydG$. 
Then $\NA(V) \simeq K \# \NA (V_{1})$ where $K=\NA (V)^{\mathrm{co}\,\NA (V_{1})}$. 
By \cite[Proposition 8.6]{heck-sch},  $K$ is the Nichols algebra of
\begin{align} \label{eq:1bpK^1}
K^1= \ad_c\NA (V_{1}) (V_2).
\end{align}
Here $K^1\in {}^{\NA (V_{1})\# \ku \Gamma}_{\NA (V_{1})\# \ku \Gamma}\mathcal{YD}$ with the adjoint action
and the coaction given by
\begin{align} \label{eq:coaction-K^1}
\delta &=(\pi _{\NA (V_{1})\#  \ku \Gamma}\otimes \id)\Delta _{\NA (V)\#  \ku \Gamma}.
\end{align}

\section{One block and one point}\label{section:block-point}
Let $(q_{ij})_{i,j \in \I_2}$, $q_{ij} \in \kut$, $a \in \ku$.
In this Section we assume that 
\begin{align}\label{eq:hyp-block-point}
q_{11} &= 1, & q_{12}q_{21} &= 1.
\end{align}
Sometimes we use $\wp = q_{12} = q_{21}^{-1}$.
Let $\lstr_{\wp}(q_{22},a)$ be the braided vector space  with  basis $(x_i)_{i\in\I_3}$ and  braiding given by
\begin{align}\label{eq:braiding-block-point}
(c(x_i \otimes x_j))_{i,j\in \I_3} &=
\begin{pmatrix}
 x_{1} \otimes x_{1}&  ( x_2 + x_{1}) \otimes x_{1}& q_{12} x_3  \otimes x_{1}
\\
 x_{1} \otimes x_2 & ( x_2 + x_{1}) \otimes x_2& q_{12} x_3  \otimes x_2
\\
q_{21} x_{1} \otimes x_3 &  q_{21}(x_2 + a x_{1}) \otimes x_3&  q_{22}x_3  \otimes x_3
\end{pmatrix}.
\end{align}
Let $V_{1} = \langle x_{1}, x_2\rangle \simeq \cV(1,2)$ (the block) and $V_2 = \langle x_3 \rangle$ (the point);
then $\lstr_{\wp}(q_{22},a)= V_{1}\oplus V_2$. For simplicity, $V =\lstr_{\wp}(q_{22},a)$.
Let $\Gamma = \Z^2$ with canonical basis $g_{1}, g_2$. Observe that $(V,c)$  can be realized in $\ydG$ via:
\begin{align}\label{eq:YD-structure-1block-1point}
\begin{aligned}
g_{1}\cdot x_{1} &= x_{1},&  g_{1}\cdot x_2 &= x_{1} + x_2, &  g_{1}\cdot x_3 &= q_{12} x_3,\\
g_2\cdot x_{1} &= q_{21}x_{1},&   g_2\cdot x_2 &= q_{21}(x_2 + ax_{1}),  & g_2\cdot x_3 &=  q_{22} x_3,\\
\deg x_1 &= g_{1},& \deg x_2 &= g_{1},&  \deg x_3 &=g_2.
\end{aligned}
\end{align}

If $a = 0$, then
$\NA (\lstr_{\wp}(q_{22},0)) \simeq \NA(V_{1}) \underline{\otimes} \NA(V_2)$,
where $\underline{\otimes}$ is the braided tensor product. 
Since $\dim \NA (V_{1})=2^4$, $\dim \NA (\lstr_{\wp}(q_{22},0)) < \infty \iff \dim \NA(\ku x_3) < \infty \iff q_{22} \in \G_{\infty}$. 
 Thus we can assume that $a \in \kut$.

\medbreak
Our main goal in this Section is to prove the following result.

\begin{theorem}\label{thm:point-block} Assume \eqref{eq:hyp-block-point} and that $a \neq 0$.
Then $ \dim \NA (\lstr_{\wp}(q_{22},a)) < \infty$ 
if and only if $q_{22} = 1$. 
Precisely,
$\dim \NA (\lstr_{\wp}(1,a))= \begin{cases} 
2^7 &\text{ if } a = 1, \\ 2^8 &\text{ if } a \in \ku \setminus \{0, 1\}. 
\end{cases}$ 
\end{theorem}

We shall apply the splitting technique cf. \S\ref{subsubsection:splitting}. To describe $K^1$, we set
\begin{align}\label{eq:zn}
z_n &:= (\ad_c x_2)^n x_3,& n&\in\N_0.
\end{align}

 We establish first a series of useful  formulae.

\begin{lemma} \label{le:-1bpz} The following formulae hold in $\NA(V)$ for all $n\in\N_0$:
\begin{align}\label{eq:block+point1}
g_{1}\cdot z_n &= q_{12}z_n,& x_{1}z_n &= q_{12}z_nx_{1},& x_{21}z_n &= q_{12}^2z_nx_{21},& 
\\ 
\label{eq:block+point2}g_2\cdot z_n &= q_{21}^nq_{22}z_n,& x_2z_n &= q_{12}z_nx_2 + z_{n+1}. 
\end{align}
\end{lemma}
\noindent\emph{Proof.}
Note that \eqref{eq:block+point1} holds for $n=0$. Indeed, $g_{1}\cdot z_0 = g_{1}\cdot x_3=q_{12}z_0$ and using derivations is easy to check that $x_{1}z_0= q_{12}z_0x_{1}$ and  $x_{21}z_0 = q_{12}^2z_0x_{21}$. Now suppose that \eqref{eq:block+point1} holds for $n$. Then, $ z_{n+1}=  (\ad_c x_2)^{n+1} x_3=  (\ad_c x_2)z_n=x_2 z_n+(g_{1} \cdot z_n) x_2=x_2z_n + q_{12}z_nx_2$.
So we compute
\begin{align*}
g_{1}\cdot z_{n+1}&=g_{1}\cdot(x_2z_n+q_{12}z_nx_2) =q_{12}(x_{1}+x_2)z_n+q_{12}^2z_n(x_{1}+x_2)\\&=q_{12}[(x_2z_n+q_{12}z_nx_2)+(x_{1}z_n+q_{12}z_nx_{1})] =q_{12}z_{n+1}.\end{align*}

Similarly,
\begin{align*}
x_{1} z_{n+1}&=x_{1}(x_2z_n+q_{12}z_nx_2)=(x_{21}+x_2x_{1})z_n+q_{12}^2z_nx_{1}x_2\\&=q_{12}^2z_nx_{21}+q_{12}x_2z_nx_{1}+q_{12}^2z_n(x_{21}+x_2x_{1})\\&=q_{12}(x_2z_n+q_{12}z_nx_2)x_{1}=q_{12}z_{n+1}x_{1}.
\end{align*}

Also, since $x_{21}x_2=x_2x_{21}+x_{21}x_{1}$ we have that
\begin{align*}
x_{21} z_{n+1}&=x_{21}(x_2z_n+q_{12}z_nx_2) =x_{21}x_2z_n+q_{12}^3z_nx_{21}x_2\\&=x_2x_{21}z_n+x_{21}x_{1}z_n+q_{12}^3z_nx_{21}x_2\\&=q_{12}^2x_2z_nx_{21}+q_{12}x_{21}z_n x_{1}+q_{12}^3z_n(x_2x_{21}+x_{21}x_{1})\\&=q_{12}^2x_2z_nx_{21}+q_{12}^3z_nx_{21}x_{1}+q_{12}^3z_nx_2x_{21}+q_{12}^3z_nx_{21}x_{1}\\&=q_{12}^2(x_2z_n+q_{12}z_nx_2)x_{21}\\&=q_{12}^2z_{n+1}x_{21}.
\end{align*}

Finally, the first equation in \eqref{eq:block+point2} follows by induction.
For $n=0$, $g_2 \cdot z_0=q_{22}z_0$. Suppose that $g_2\cdot z_n=q_{21}^nq_{22}z_n$. Then,
\begin{align*}
g_2\cdot z_{n+1}&=g_2\cdot(x_2z_n+q_{12}z_nx_2) \\&= q_{21}(x_2+ax_{1})(q_{21}^nq_{22}z_n)+q_{12}(q_{21}^nq_{22}z_n)q_{21}(x_2+ax_{1}) \\&= q_{21}^{n+1}q_{22}(x_2z_n+q_{12}z_nx_2)+aq_{21}^{n+1}q_{22}(x_{1}z_n+q_{12}z_nx_{1})\\&=q_{21}^{n+1}q_{22}z_{n+1}. 
\hspace{200pt}\qed
\end{align*}

We define 
\begin{align*} \mu_0 &=1, &\mu_{1} &=a,&  \mu_{2} &= a, & \mu_{3} &= a(a+1), \\
y_0 &= 1, &y_{1} &= x_{1}, &  y_{2} &= x_{21}, & y_{3} &= x_{1}x_{21}.
\end{align*}

\begin{lemma}\label{lemma:derivations-zn}
For all $k \in\N_0$, $\partial_{1}(z_k) =\partial_2(z_k) =  0$, and 
\begin{align*}
&\partial_3(z_{k})=\mu_ky_k, \quad k\in\I_{0,3},&  &\partial_3(z_{k})=0,\quad k\geq 4.&
\end{align*}
\end{lemma}
\noindent\emph{Proof.} 
Clearly, $\partial_{1}(z_0) =\partial_2(z_0) =  0$, $\partial_3(z_0) = 1$. 
Recursively, $\partial_{1}(z_k) =  0$ for all $k$. 
If $\partial_2(z_k) =  0$, then $\partial_2(z_{k+1}) = \partial_2(x_2z_k+q_{12}z_kx_2)=g_{1} \cdot z_k+q_{12}z_k \overset{\eqref{eq:block+point1}}{=}0$. Next,
\begin{align*}
\partial_3(z_{1})&=\partial_3(x_2x_3+q_{12}x_3x_2)=x_2+q_{12}(q_{21}(x_2+ax_{1}))=ax_{1}=\mu_{1}y_{1},\\
\partial_3(z_2)&=\partial_3(x_2z_{1}+q_{12}z_{1}x_2)=ax_2x_{1}+q_{12}ax_{1}q_{21}(x_2+ax_{1})=ax_{21}=\mu_2y_2, \\
\partial_3(z_3)&=\partial_3(x_2z_2+q_{12}z_2x_2)=ax_2x_{21}+q_{12}ax_{21}q_{21}(x_2+ax_{1})
\\&=
ax_2(x_2x_{1}+x_{1}x_2)+a(x_2x_{1}+x_{1}x_2)(x_2+ax_{1})\\&=ax_2^2x_{1}+ax_{1}x_2^2+a^2x_{1}x_2x_{1}\stackrel{\eqref{eq-jordan-plane-char2}}{=}(a+a^2)x_{1}x_2x_{1}\\&=(a+a^2)x_{1}x_{21}=\mu_3y_3,\\
\partial_3(z_4)&=\partial_3(x_2z_3+q_{12}z_3x_2)\\&=(a+a^2)x_2x_{1}x_2x_{1}+q_{12}(a+a^2)x_{1}x_2x_{1}(g_2\cdot x_2)\\&=(a+a^2)x_2x_{1}x_2x_{1}+q_{12}(a+a^2)x_{1}x_2x_{1}(q_{21}(x_2+ax_{1}))\\&=(a+a^2)(x_2x_{1}x_2x_{1}+x_{1}x_2x_{1}x_2)\stackrel{\eqref{eq-jordan-plane-char2}}{=}0. \hspace{110pt}\qed
\end{align*}

\begin{lemma} \label{K-basis} Let $B_{1}:=\{z_i: i \in \I_{0,2}\}$ and $B_2:=\{z_i: i \in \I_{0,3}\}$. If $a=1$ (resp. $a\neq 1$), then $B_{1}$ (resp. $B_2$) is a basis of $K^1$.
\end{lemma}
\pf Notice that $(\ad_c x_{1})z_n=0$ and $(\ad_c x_{21})z_n=0$. By Theorem  \ref{thm:jordan-plane-char2} and Lemma \ref{lemma:derivations-zn}, if $a=1$ (resp. $a\neq 1$), then $B_{1}$ (resp. $B_2$) generates $K^1$. Since the elements of $B_i$ ($i \in \I_{0,2}$) are homogeneous of distinct degrees and are non-zero, it follows that $B_i$ ($i \in \I_{0,2}$) is a linearly independent set.   \epf

Let $i\in \N_0$. We define recursively the scalars $\nu_{i,j}$, for $j>i$, by
 \begin{align*}\nu_{i,i} &= 1, & \nu_{i,j} &= (a+(j-1)) \, \nu_{i,j-1}. \end{align*}

\begin{lemma} \label{le:zcoact} The coaction \eqref{eq:coaction-K^1} on $z_i$, $i\in\I_{0,3}$, is given,
(for $n=0,1$)  by 
\begin{align*} 
\delta (z_{2n}) &=
\sum _{k=1}^n \nu_{k,n} 
x_{1}x_{21}^{n-k}g_{1}^{2k-1}g_2\otimes z_{2k-1}
  + \sum _{k=0}^n \nu_{k,n} x_{21}^{n-k}g_{1}^{2k}g_2\otimes z_{2k},
\\ 
\delta (z_{2n+1}) &=
\sum _{k=0}^n \nu_{k,n+1}  x_{1}x_{21}^{n-k}g_{1}^{2k}g_2\otimes z_{2k}
 + \sum _{k=0}^n \nu_{k+1,n+1} x_{21}^{n-k}g_{1}^{2k+1}g_2\otimes z_{2k+1}.
\end{align*} 
\end{lemma} 

\pf Similar to the proof of \cite[Lemma 4.2.5]{aah-triang}.
\epf

Lemma \ref{le:zcoact} implies that $K^1$ is of diagonal type with braiding given by
\begin{align}\label{braid-bloc+1ponto}
c(z_i \otimes z_j)=q_{21}^{j-i}q_{22}z_j \otimes z_i, \qquad \quad \forall\,\, i,j.
\end{align}

\smallbreak
 
 Now we are ready for to prove the main result of this Section. 

 \smallbreak
\noindent{\it Proof of Theorem \ref{thm:point-block}.}
If $q_{22}=1$, then the Dynkin diagram of $K^1$ is totally disconnected with vertices labelled with $1$. 
Thus, if $a=1$  then $\dim\toba(K^1)=2^3$ and $\dim \NA (\lstr_{\wp}(1,1))=2^7$; if $a\neq 1$, 
then  $\dim\toba(K^1) = 2^4$  and  $\dim \NA (\lstr_{\wp}(1,a))=2^8$. 
 If $q_{22} \neq 1$, then the Dynkin diagram of $K^1$ is 
\begin{align*} a&=1:& &\xymatrix{  &\circ^{q_{22}}\ar@{-}[rd]^{q_{22}^2} &\\ 
\circ^{q_{22}}\ar@{-}[ru]^{q_{22}^2} \ar@{-}[rr]^{q_{22}^2} &  &\circ^{q_{22}}; } & a&\neq 1: 
&\xymatrix{&\ar@{-}[ld]_{q_{22}^2} \ar@{-}[rd]^{q_{22}^2}\circ^{q_{22}} &  \\ \circ^{q_{22}}\ar@{-}@/_1.5pc/[rr]_{q_{22}^2}\ar@{-}[r]^{q_{22}^2} & \circ^{q_{22}} \ar@{-}[u]^{q_{22}^2}\ar@{-}[r]^{q_{22}^2}& \circ^{q_{22}}.} 
\end{align*}
By inspection of the lists in \cite{wang-rank3,wang-rank4} we conclude that $\dim \NA(K^{1}) = \infty$.
\qed

\subsection{The presentation by generators and relations}\label{subsubsection:presentation-block-point}
Let $c$ be the braiding of $K^1$ as in \eqref{braid-bloc+1ponto}.
Then  $q_{22}=1$ if and only if $c^2=\id$. 
Hence, for $a=1$ (resp. $a\neq 1$), $\NA (K^1)$ is the algebra generated by $ z_0,z_{1},z_2$ (resp. $z_0,z_{1},z_2,z_3$) with relations
\begin{align*}
z_i^2=0,\qquad z_i z_j= q_{21}^{j-i}z_jz_i, \qquad i\neq j.
\end{align*}
Thus, we have the following results.

\begin{prop}\label{nich-ex1} The algebra
$\toba(\lstr_{\wp}(1,1))$ is presented by generators $x_{1}$, $x_2$, $x_3$ 
with defining relations \eqref{eq-jordan-plane-char2} and 
\begin{align}
&x_{1}z_j= q_{12} \, z_jx_{1}, \quad z_{j+1}=x_{2}z_j+q_{12} z_jx_{2},\qquad j \in \N_0, \\
z_i&z_j=q_{21}^{j-i}z_jz_i,\quad  z_j^2=0,\quad z_k=0,\qquad i,j \in \I_{0,2},\quad k\geq 3.
\end{align} 
The dimension of $\toba(\lstr_{\wp}(1,1))$ is $2^{7}$, since it has a PBW-basis
\begin{align*}
\{ x_{1}^{m_{1}}x_{21}^{m_2} x_{2}^{m_3} z_{2}^{n_{2}}z_{1}^{n_{1}} z_0^{n_0}\,:\, m_{1}, m_2,n_i \in\I_{0,1},\, m_3 \in \I_{0,3} \}.
\hspace{30pt}\qed
\end{align*} 
\end{prop}

\begin{prop}\label{nich-ex2} The algebra $\toba(\lstr_{\wp}(1,a))$, $a\neq 1$, is presented by 
generators $x_{1},x_2, x_3$ with defining relations \eqref{eq-jordan-plane-char2} and
\begin{align}
&x_{1}z_j= \wp \, z_jx_{1},\quad z_{j+1}=x_{2}z_j+\wp z_jx_{2}, \qquad j \in \N_0, \\
 z_i&z_j=\wp^{i-j}z_jz_i,\quad z_j^2=0,\quad z_k=0, \qquad i,j \in \I_{0,3},\quad k \geq  4. 
\end{align}
The dimension of $\toba(\lstr_{\wp}(1,a))$ is $2^{8}$, since it has a PBW-basis
\begin{align*}
\{ x_{1}^{m_{1}}x_{21}^{m_2} x_{2}^{m_3} z_{3}^{n_3} z_{2}^{n_{2}}z_{1}^{n_{1}} z_0^{n_0}\,:\, m_{1}, m_2,n_i \in\I_{0,1},\, m_3 \in \I_{0,3} \}. \hspace{30pt}\qed
\end{align*}
\end{prop}

\subsection{Realizations}\label{subsec:realization-block-pt}
Let $(g_{1}, \chi_{1}, \eta)$ be a YD-triple and $(g_2, \chi_2)$
a YD-pair for $H$, see \S \ref{subsec:realizations-block}. 
Let $(V, c)$ be a braided vector space  with braiding 
\eqref{eq:braiding-block-point}.
Then  $\cV_{g_{1}}(\chi_{1}, \eta) \oplus  \ku_{g_2}^{\chi_2} \in \ydh$
is a \emph{principal realization} of $(V, c)$  over  $H$ if
\begin{align*}
q_{ij}&= \chi_j(g_i),& &i, j\in \I_{2};& a&= q_{21}^{-1}\eta(g_2).
\end{align*}
Thus $(V, c) \simeq \cV_{g_{1}}(\chi_{1}, \eta) \oplus  \ku_{g_2}^{\chi_2}$ 
as braided vector space.
Hence, if $H$ is finite-dimensional and $(V, c)\simeq \lstr_{\wp}(1,a)$, $a\neq 0$,
then $\toba \big(\cV_{g_{1}}(\chi_{1}, \eta) \oplus  \ku_{g_2}^{\chi_2}\big) \# H$ is a finite-dimensional Hopf algebra.
Observe that the existence of a YD-triple for $H$ finite-dimensional  is not granted; for instance, 
$\wp = q_{12}$ should be a root of 1, otherwise there is no such triple.
Suppose that $\ord \wp = M \in \N$. Notice that $M$ is odd because $\car\ku=2$.
Here are some explicit examples of finite-dimensional pointed Hopf algebras 
like this: take $\Gamma = \langle g_{1} \rangle \times \langle g_2 \rangle$ where both $g_{1}$ and $g_2$ have order 
$2M$. 
Then $(V, c)$ is realized in
$\ydG$ with structure as in \eqref{eq:YD-structure-1block-1point}
and   $\dim \NA(V) \# \ku \Gamma = 2^9M^2$ (if $a=1$) or  $2^{10}M^2$ (if $a \neq 1$).

\section{One block and several points}\label{section:block-points}

Let $\theta \in \N_{\ge 3}$,  
$\Iw_\theta = \I_{\theta} \cup \{\fudos\}$; as usual $\lfloor x\rfloor$ is the integral part of $x \in \mathbb R$.
We fix a matrix $\bq = (q_{ij})_{i,j \in \I_{\theta}}$ with entries in  $\kut$ and 
$\af=(1, a_2, \dots, a_\theta) \in \ku^{\theta}$.
We assume that 
\begin{align} \label{eq:hyp-block-points}
q_{11} &= 1,& q_{1j}q_{j1}&=1, \text{ for all } j \in \I_{2,\theta}, &
\af & \neq (1, 0, \dots, 0).
\end{align}
Let $(V, c)$ be the braided vector space of dimension $\theta + 1$,
with a basis $(x_i)_{i\in\Iw_{\theta}}$ and braiding given  by 
\begin{align}\label{eq:braiding-block-several-point}
c(x_i \otimes x_j) &= \begin{cases}
q_{\lfloor i\rfloor j} x_j  \otimes x_i, &i\in \Iw_{\theta},\, j\in \I_{\theta};\\
q_{\lfloor i\rfloor 1} (x_{\fudos} + a_{\lfloor i\rfloor} x_{1}) \otimes x_{i}, & i\in \Iw_{\theta},\, j =\fudos.
\end{cases}
\end{align}
Then $V = V_{1} \oplus V_{2}$ where 
$V_{1} = \langle x_{1}, x_{\fudos} \rangle \simeq \cV(1,2)$ (the block) and $V_2 = \langle x_2, \dots, x_{\theta} \rangle$ (the points). If $\Gamma = \zt$ with basis $(g_h)_{h\in \I_{\theta}}$, then $V$ can be realized in $\ydG$   as in \eqref{eq:YD-structure-1block-1point}.
Here is the main result of this Section.

\begin{theorem}\label{thm:points-block} Assume \eqref{eq:hyp-block-points}.
Then $ \dim \NA (V) = \infty$. 
\end{theorem}

We shall use the material from the previous Section  with $\fudos$ replacing 2
for instance $x_{\fudos 1} = x_{\fudos}x_{1} + x_{1} x_{\fudos}$.
We shall apply the splitting technique cf. \S\ref{subsubsection:splitting}. To describe $K^1$, we introduce the elements
\begin{align}\label{eq:zin}
z_{i,n} &:= (ad_c x_{\fudos})^n x_i,& i&\in\I_{2, \theta},& n&\in\N_0.
\end{align}

Let $i\in \I_{2,\theta}$, $n\in \N_0$. By Lemma \ref{le:-1bpz}, we have that 
\begin{align}
\label{eq:1block+points-action}
 g_{1}\cdot z_{i,n} &= q_{1i}z_{i,n},& z_{i,n+1} &= x_{\fudos}z_{i,n}+q_{1i}z_{i,n}x_{\fudos},
& x_{1} z_{i,n} &= q_{1,i}z_{i,n}x_{1}.
\end{align}

Consequently, 
\begin{align}\label{eq:1block+points-bis}
g_h\cdot z_{i,n} &= q_{h1}^nq_{hi}z_{i,n}, & h &\in \I_{2, \theta}.
\end{align}
In fact, $g_h \cdot z_{i,0}=g_h \cdot x_i =q_{hi} x_i$. Suppose that $g_h \cdot z_{i,n} = q_{h1}^nq_{hi}z_{i,n}$. Thus,
\begin{align*} g_h\cdot z_{i,n+1}&=g_h\cdot(x_{\fudos}z_{i,n}+q_{1i}z_{i,n}x_{\fudos})\\&= q_{h1}(x_{\fudos}+a_hx_{1})q_{h1}^nq_{hi}z_{i,n}+q_{1i}q_{h1}^{n+1}q_{hi}z_{i,n}(x_{\fudos}+a_hx_{1})\\&= q_{h1}^{n+1}q_{hi}(x_{\fudos}z_{i,n}+q_{1i}z_{i,n}x_{\fudos})= q_{h1}^{n+1}q_{hi}z_{i,n+1}.\end{align*}

As in Lemma \ref{lemma:derivations-zn}, we define for $i \in\I_{2, \theta}$,
\begin{align*} \mu_0^{(i)} &=1, &\mu_{1}^{(i)} &=a_i,&  \mu_{2}^{(i)} &= a_i, & \mu_{3}^{(i)} &= a_i(a_i+1), \\
y_0 &= 1, &y_{1} &= x_{1}, &  y_{2} &= x_{\fudos 1}, & y_{3} &= x_{1}x_{\fudos 1}.
\end{align*}
Hence $\partial_h(z_{i,n})=0$ for $i \in\I_{2, \theta}$, $n\in\N_0$, $i \neq h \in \Iw_{\theta}$ and
\begin{align*}
&\partial_i(z_{i,n})=\mu_n^{(i)} y_n, \quad n\in\I_{0,3},&  &\partial_i(z_{i,n})=0,\quad n\geq 4.&
\end{align*}

For $i \in \I_{2,\theta}$, we define
\begin{align*} J_i&= \begin{cases}
\{(i,0)\},& a_i=0,\\
\{(i,0),(i,1),(i,2)\}, &  a_i=1,\\
\{(i,0),(i,1),(i,2), (i,3)\}, &  a_i\notin \{0, \, 1\}, 
\end{cases} & J &=  \bigcup_{i \in \I_{2,\theta} } J_i.
\end{align*}

\begin{lemma}\label{lemma:braiding-K-block-points} 
The family $B = (z_{i,n})_{(i,n) \in J}$ is a basis of the braided vector space $K^1$,
which is of diagonal type with braiding
\begin{align}\label{braid:block+points}
c(z_{i,m} \otimes z_{j,n} ) &= q_{i1}^nq_{1j}^m  q_{ij} z_{j,n} \otimes z_{i,m},& 
(i,m), (j,n) &\in J.
\end{align}
\end{lemma}

\pf Arguing as in Lemma \ref{K-basis}, we see that $B$ is a basis. 
We compute the coaction \eqref{eq:coaction-K^1} on $z_{j,n}$ as in Lemma \ref{le:zcoact}
and then \eqref{braid:block+points} follows.
\epf

\noindent{\it Proof of Theorem \ref{thm:points-block}.}
It is enough to show that $\dim \toba(K^1) = \infty$.
By \eqref{braid:block+points} we may assume that the Dynkin diagram of the  matrix 
$(q_{ij})_{i, j \in \I_{2, \theta}}$ is connected.
We then may assume that $\theta=3$ by taking a suitable sub-diagram; thus 
$\widetilde q_{23} := q_{23} q_{32} \neq 0$.
We distinguish then three cases.
First assume $\af=(1, a, 0)$ with $a \neq 0$. By Theorem \ref{thm:point-block} 
applied to $V_{1} \oplus \langle x_2 \rangle$, $q_{22}=1$.  By Lemma \ref{lemma:braiding-K-block-points}, 
$K^1$ is of diagonal type. If $a = 1$, then  its  Dynkin diagram is
\begin{align*}
\xymatrix@R-5pt{& &  \circ^{1} \\ 
	\circ^{1}\ar@{-}[r]^{\widetilde q_{23}} & \circ^{q_{33}} \ar@{-}[ru]^{\widetilde q_{23}}\ar@{-}[r]^{\widetilde q_{23}} &   \circ^{1} }
\end{align*}        
\noindent which does not appear in the list in \cite{wang-rank4}. 
If $a \neq 1$, then the diagram above appears a sub-diagram.
The case $\af=(1,0,b)$ with $b \neq 0$ is similar. 
As well, if $\af=(1,a,b)$ with $a, b \neq 0$, then the diagram above also appears a sub-diagram of that of $K^1$.
\qed

\section{Several blocks and one point}\label{section:blocks-point}

Let $t\geq 2$ and $\theta = t +1$. As in \cite{aah-triang,aah-oddchar} we use  the notation:
\begin{align*}
\Idd_k &= \{k, k + \tfrac{1}{2}\},& k&\in \I_t;
& \Idd &= \Idd_{1} \cup \dots \cup \Idd_{t} \cup \{\theta\}.
\end{align*}
We fix a matrix $\bq = (q_{ij})_{i,j \in \I_{\theta}}$ with entries in  $\kut$
and $\ba=(a_{1},\ldots,a_t) \in \ku^{t}$. We assume that 
\begin{align} \label{eq:hyp-blocks-point}
q_{ii} &=1,& q_{ij}q_{ji} &=1, & \text{for all }  i\neq j &\in \I_{\theta}; & a_j &\neq 0, &j &\in \I_t.
\end{align}

Let  $\pos(\bq,\ba)$ be the braided vector space with basis $(x_i)_{i\in\Idd}$ and braiding
\begin{align}\label{eq:braiding-several-blocks-1pt}
c(x_i\ot x_j) &= \left\{ \begin{array}{ll}
q_{ \lfloor i \rfloor   \lfloor j\rfloor} \, x_j\otimes x_i, & \lfloor i\rfloor \leq t, \, \lfloor i\rfloor \neq \lfloor j\rfloor, \\
 \, x_j\otimes x_i, & \lfloor i\rfloor =j \leq t, \\
( \, x_j+x_{\lfloor j\rfloor})\otimes x_i, & \lfloor i\rfloor \leq t, \, j=\lfloor i\rfloor +\frac{1}{2}, \\
q_{\theta j} \, x_j\otimes x_{\theta}, & i = \theta, \, j\in\I_{\theta}, \\
q_{\theta \lfloor j\rfloor} \, (x_j+a_{\lfloor j\rfloor} x_{\lfloor j\rfloor})\otimes x_{\theta}, & i = \theta, \, j\notin\I_{\theta}.
\end{array} \right.
\end{align}
Let $V_1=W_{1}\oplus\ldots\oplus W_t$ where
$W_{k} = \langle x_{k}, x_{\fkdos} \rangle \simeq \cV(1,2)$ (the blocks);
and let  $V_2 = \langle x_{\theta} \rangle$ (the point).
Then $\pos(\bq,\ba) = V_{1} \oplus V_{2}$. If $\Gamma = \zt$ 
with basis $(g_i)_{i\in \I_{\theta}}$, then there is an action of $\Gamma$ on $V$ determined by  
\begin{align}\label{eq:YD-structure-blocks-1point}
c(x_i\ot x_j) &= g_{i}\cdot x_{j} \otimes x_i,& i&\in \I_{\theta}, \ j \in \Idd. 
\end{align}
Thus $V$ is realized in $\ydG$ with the grading $\deg(x_i) = g_{\lfloor i \rfloor}$, $i \in \Idd$.

\medbreak
Here is the main result of this Section; see \eqref{eq:dim-poseydon} for the explicit formula of the  dimension.

\begin{theorem}\label{teo-blocks-point}
Assume \eqref{eq:hyp-blocks-point}. Then $ \dim \NA (\pos(\bq,\ba)) < \infty$. 
\end{theorem}

Let $j\in\I_t$. We set $x_{j+\frac{1}{2} \, j}= x_{j+\frac{1}{2}} x_j + x_j x_{j+\frac{1}{2}}$ and define 
\begin{align*}
\mu_0^{(j)} &=1, &\mu_{1}^{(j)} &=a_j,&  \mu_{2}^{(j)} &= a_j, & \mu_{3}^{(j)} &= a_j(a_j+1),  & \mu_{n}^{(j)} &= 0\,\,\,\text{ if }n\geq 4, 
\\
y_{j,0} &= 1, &y_{j,1} &= x_{j}, &  y_{j,2} &= x_{j + \frac{1}{2} \, j}, & y_{j,3} &= x_{j}x_{j + \frac{1}{2} \, j}, & y_{j,n} &= 0\,\,\,\text{ if }n\geq 4.
\end{align*}

To apply the splitting technique, see \S\ref{subsubsection:splitting}, we introduce the elements
\begin{align}
\label{eq:defn-sch-several-blocks}
\sch_{\bn} &:= (\ad_c x_{\fudos})^{n_{1}} \dots (\ad_c x_{t+\frac{1}{2}})^{n_{t}} x_{\theta}, & & \bn = (n_{1},\dots,n_{t})\in\N^t_0.
\end{align}

We start establishing some useful formulas.

\begin{lemma}\label{lem1-blocks-point} Let $j\in \I_t$ and  $\bn = (n_{1},\dots,n_{t})\in\N^t_0$. Then 
	\begin{align}
	&\quad \ad_c x_j(\sch_{\bn})=\ad_c x_{j+\frac{1}{2}\,j}(\sch_{\bn})=0 \label{eq1-lem},\\[.8em] 
	&\qquad\ad_c x_{j+\frac{1}{2}}(\sch_{\bn})=\prod\limits_{i<j}q_{ji}^{n_i}\sch_{\bn+\mathbf{e}_j}, \label{eq2-lem}\\
	&g_j\cdot \sch_{\bn}=q_{j\theta}\prod\limits_{i=1}^t q_{ji}^{n_i},\quad g_{\theta}\cdot \sch_{\bn}=\prod\limits_{i=1}^t q_{\theta i}^{n_i} \sch_{\bn} \label{eq3-lem}\\
	\partial_j(\sch_{\bn})&=\partial_{j+1}(\sch_{\bn})=0,\qquad  \partial_{\theta}(\sch_{\bn})=\prod\limits_{i=1}^t\mu_{n_i}^{(i)}y_{1,n_{1}}\ldots y_{t, n_t}.\label{eq4-lem}
	\end{align}
\end{lemma}
\begin{proof} Similar to the proof of \cite[Lemma 7.2.3]{aah-triang}.
\end{proof}

Let us set
\begin{align*}
b_j &:=2, \text{ if } a_j =1, & b_j &:=3, \text{ if } a_j \neq  1, &
\text{and }  \mathbf{b}=(b_{1},\ldots,b_t)\in \N^t.
\end{align*}

Arguing as in \cite[$\S 7.2$]{aah-triang}, we conclude from
Lemma \ref{lem1-blocks-point}:

\begin{lemma}\label{lem-basis-blocks-point} 
Let $\mathcal{A}=\{\bn\in \N_0^t\,:\, \bn\leq\mathbf{b}\}$ ordered lexicographically.
\begin{enumerate}[leftmargin=*,label=\rm{(\roman*)}]
\item  The elements $(\sch_{\bn})_{\bn \in \mathcal{A}}$ form a basis of $K^1$. \smallbreak
\item The coaction \eqref{eq:coaction-K^1} on $\sch_{\bn}$ is given by
\begin{align*}
\delta(\sch_{\bn})=\sum\limits_{0\leq \mathbf{k}\leq \bn}\nu_{\mathbf{k}}^{\bn} y_{1, n_{1}-k_{1}}\ldots y_{t, n_t-k_t}g_{1}^{k_{1}}\ldots g_t^{k_t}g_{\theta}\otimes \sch_{\mathbf{k}} 
\end{align*}
for some scalars $\nu_{\mathbf{k}}^{\bn}$, $0\leq \mathbf{k}\leq \bn$,  with $\nu_{\bn}^{\bn}=1$. \smallbreak
\item The braided vector space $K^1$ is of diagonal type with respect to the basis $(\sch_{\bn})_{\bn \in \mathcal{A}}$ with matrix braiding $(p_{\bm,\bn})_{\bm,\bn\in \mathcal{A}}$, where
\[p_{\bm,\bn}=\prod\limits_{i,j=1}^t q_{ij}^{m_in_j}q_{i\theta}^{m_i}q_{\theta j}^{n_j}.\]
Hence, the corresponding generalized Dynkin diagram has labels
\begin{align*}
&p_{\bm,\bm}=1& &p_{\bm,\bn}p_{\bn,\bm}=1,& &\bm\neq \bn.&
\end{align*}
\end{enumerate}
\end{lemma}

\noindent {\it Proof of Theorem \ref{teo-blocks-point}.} By Lemma \ref{lem-basis-blocks-point}, $\dim \toba(K^1)=2^{|\mathcal{A}|}$. 
Now the blocks $W_i$ and $W_j$, $i\neq j$, commute in the braided sense by definition, therefore 
$\toba(V_1) \simeq \toba(W_1) \underline{\otimes} \toba(W_2) \dots \underline{\otimes} \toba(W_1)$.
Hence
\begin{align}\label{eq:dim-poseydon}
\dim \toba(\pos(\bq,\ba))=2^{4t+|\mathcal{A}|}. 
\end{align}
\qed

\subsection{The presentation by generators and relations}\label{subsubsection:presentation-blocks-point}

\begin{prop}\label{generatores-relations-blocks-point}
The algebra $\toba(\pos(\bq,\ba))$ is presented by generators $x_i$, $i\in \Idd$, and relations
\begin{align}
&x_i^2=0,\qquad\quad  x_{i+\frac{1}{2}}^4=0,& & i\in \I_t,&\\
&x_{i+\frac{1}{2}}^2x_i+x_ix_{i+\frac{1}{2}}^2+x_ix_{i+\frac{1}{2}}x_i=0,&& i\in \I_t,&\\
&x_ix_{i+\frac{1}{2}}x_ix_{i+\frac{1}{2}}+x_{i+\frac{1}{2}}x_ix_{i+\frac{1}{2}}x_i=0,&&  i\in \I_t,&\\
&x_ix_j=q_{\lfloor i\rfloor\lfloor j\rfloor}x_jx_i,& &\lfloor i\rfloor \neq \lfloor j\rfloor  \in\I_t,&\\
&x_ix_{\theta}=q_{i\theta}x_{\theta}x_i,&& i\in \I_t,&\\
&\left(\ad_c x_{i+\frac{1}{2}}\right)^{1+b_i}(x_{\theta})=0,&& i\in \I_t,&\\
&\sch_{\bm}\sch_{\bn}=p_{\bm,\bn}\sch_{\bn}\sch_{\bm},&&\bm\neq \bn\in \mathcal{A},&\\
&\sch_{\bn}^2=0,& &\bn\in \mathcal{A}.&
\end{align}
A basis of $\toba(\pos(\bq,\ba))$ is given by
\[B=\{y_{1,m_{1}}x_{\frac{3}{2}}^{m_2}\ldots y_{t,m_{2t-1}}x_{t+\frac{1}{2}}^{m_{2t}}\prod\limits_{\bn\in \mathcal{A}}\sch_{\bn}^{b_{\bn}}\,:\, 0\leq b_{\bn}<2,\,\,0\leq m_i<4 \}.\]
Hence $\dim\toba(\pos(\bq,\ba))=2^{4t+|\mathcal{A}|}$. \qed
\end{prop}

\subsection{Realizations}\label{subsec:realizations-blocks-pt}
Let $H$ be a Hopf algebra, $(g_i, \chi_i, \eta_i)$, $i\in \I_{t}$,
a family of YD-triples and $(g_{\theta}, \chi_{\theta})$ a YD-pair for $H$, see \S \ref{subsec:realizations-block}. 
Let $(V, c)$ be a braided vector space  with braiding 
\eqref{eq:braiding-several-blocks-1pt}.
Then  
\begin{align}
\cV := \Big(\oplus_{i\in \I_{t}} \cV_{g_i}(\chi_i, \eta_i)\Big) \oplus  \ku_{g_\theta}^{\chi_\theta}  \in \ydh
\end{align}
is a \emph{principal realization} of $(V, c)$  over  $H$ if
\begin{align*}
q_{ij}&= \chi_j(g_i),& &i, j\in \I_{\theta};& a_j&= q_{j1}^{-1}\eta_j(g_j),\, j\in \I_{t}.
\end{align*}
Consequently, if $H$ is finite-dimensional,
then so is $\toba (\cV) \# H$. But the existence of such $H$ requires that all $q_{ij}$'s are roots of 1.
In this case, let $\Gamma=(\Z/N)^\theta$ where $N$ is even and divisible by $\ord q_{ij}$ for all $i,j$.
Then  $(V, c)$ is realized in $\ydG$ with action \eqref{eq:YD-structure-blocks-1point}.
 Thus $\toba(\pos(\bq,\ba))\# \Bbbk \Gamma$ is a pointed Hopf algebra of dimension $2^{4t+|\mathcal{A}|} N^\theta$.

\section{One pale block and one point}\label{sec:paleblock}
An indecomposable Yetter-Drinfeld module which is decomposable as brai\-ded vector space is called a \emph{pale block}
\cite{aam}; the simplest examples were studied in \cite{aah-triang,aah-oddchar}. 
We extend the analysis there to characteristic 2.

Let $(q_{ij})_{i,j \in \I_2}$ be a matrix with non-zero entries; we assume that $q_{11} =1$
and $q_{12}q_{21} = 1$; we set $\wp = q_{12}=q_{21}^{-1}$. 
Let $V = \eny_{\wp}(q_{22})$ be the braided vector space of dimension 3
with basis $(x_i)_{i\in\I_3}$ and braiding given by
\begin{align}\label{eq:braiding-paleblock-point}
(c(x_i \otimes x_j))_{i,j\in \I_3} &=
\begin{pmatrix}
 x_{1} \otimes x_{1}&   x_2  \otimes x_{1}& q_{12} x_3  \otimes x_{1}
\\
 x_{1} \otimes x_2 &  x_2  \otimes x_2& q_{12} x_3  \otimes x_2
\\
q_{21} x_{1} \otimes x_3 &  q_{21}(x_2 +  x_{1}) \otimes x_3& q_{22} x_3  \otimes x_3
\end{pmatrix}.
\end{align}

Let $V_{1} = \langle x_{1}, x_2\rangle$ (the pale block), 
$V_2 = \langle x_3\rangle$ (the point) and $\Gamma = \Z^2$ with a basis $g_{1}, g_2$.
Notice that $\toba(V_1)$ is a truncated symmetric algebra of dimension 4.
We realize $V$ in $\ydG$ by  $\deg x_{1} = \deg x_{2}=g_{1}$, $\deg x_3 = g_2$, 
\begin{align}
\label{eq:action-pale}
\begin{aligned}
&g_{1}\cdot x_{1} = x_{1},& &  g_{1}\cdot x_2 = x_2, & & g_{1}\cdot x_3 = q_{12} x_3,\\
&g_2\cdot x_{1} = q_{21}x_{1},& &  g_2\cdot x_2 = q_{21}(x_2 + x_{1}), & & g_2\cdot x_3 =  q_{22} x_3.
\end{aligned}
\end{align}

\begin{theorem}\label{teo-paleblock}
The Nichols algebra $\toba(\eny_{\wp}(q_{22}))$ is finite-dimensional if and only if $q_{22}=1$ or $q_{22}=\omega$, with $\omega\in \G'_3$.
\end{theorem}

To apply the splitting technique, see \S\ref{subsubsection:splitting}, we introduce the elements
\begin{align*}
&\sh_{m, n} = (\ad_{c} x_{1})^m(\ad_{c} x_{2})^n x_{3},& &w_m = \sh_{m,0},& &z_n=\sh_{0, n},& &\quad m,n\in \N_0.&
\end{align*}

By direct computation
\begin{align}\label{eq:paleblock-1}
g_{1}\, \cdot \, &\sh_{m, n}  =  q_{12} \sh_{m, n}, & g_2\cdot w_m  &=  q_{21}^{m}q_{22} w_m,\\[.2em]
\label{eq:paleblock-2}
z_{n+1} &= x_2z_n + q_{12} z_n x_2,& \sh_{m + 1, n} &= x_{1}\sh_{m, n} + q_{12}\sh_{m, n} x_{1},\\[.2em]
\label{eq:paleblock-3}
\partial_{1}(&\sh_{m, n})=\partial_2(\sh_{m, n})=0,&  \partial_3(w_{m})& =0, \text{ for all } m>0.
\end{align}

Since $x_{1}$ and $x_2$ commute, $\sh_{m, n}=(\ad_c x_2)^n(\sh_{m, 0})=(\ad_c x_2)^n(w_m)$. By \eqref{eq:paleblock-3} $w_m=0$ and thus $\sh_{m, n}=0$, for all $m>0$. Hence $\{z_n\,:\,n\in \N_0\}$ generates $K^1$. It is easy to check that
\begin{align}\label{aux-for}
&g_2\cdot z_n=q_{21}^nq_{22}z_n,\quad &  &\quad \partial_3(z_n)=x_{1}^n,\quad n\in \N_0.&
\end{align}
As $x_{1}^2=0$ we conclude that $\{z_0,z_{1}\}$ is a basis of $K^1$. The coaction is given by $\delta(z_0)=g_2\otimes z_0$ and $\delta(z_{1})=x_{1}g_2\otimes z_0+g_{1}g_2\otimes z_{1}$. From \eqref{aux-for} follows that $K^1$ is a braided vector space of diagonal type with braiding 
\[c(z_i\otimes z_j)=q^{j-i}_{21}q_{22}z_j\otimes z_i,\qquad i,j\in \I_{0,1}.\]

\noindent{\it Proof of Theorem \ref{teo-paleblock}.} If $q_{22}=1$, then the Dynkin diagram of $K^1$ is totally disconnected with vertices labelled with $q_{22}$. In this case $z_0^2=z_{1}^2=0$, $\dim\toba(K^{1})=4$ and so $\dim\toba(\eny_{\wp}(1))=2^4$.
If $q_{22}\neq 1$, the Dynkin diagram of $K^1$ is 

\centerline{\xymatrix{ &\circ^{q_{22}}\ar@{-}[r]^{q_{22}^2} & \circ^{q_{22}} .}}

\noindent By inspection in the list of \cite{heck-wang}, $\dim \toba(K^1)<\infty$ if and only if $q_{22}=\omega$, with $\omega\in \G'_3$. In this case, $\dim\toba(K^1)=3^3$ and so $\dim\toba(\eny_{\wp}(\omega))=2^{2}3^3$. \qed

\subsection{The presentation by generators and relations}\label{subsubsection:presentation-pale-block}

\begin{prop}\label{nich-pale-1} The algebra
 $\toba(\eny_{\wp}(1))$ is presented by generators $x_{1},x_2, x_3$ with defining relations
\begin{align}
&x_{1}^2= 0,\qquad  x_{2}^2=0,\qquad x_{1}x_2=x_2x_{1}, \\
x_{1}x_3&=\wp x_3x_{1},\quad z_{1}=x_2x_3+\wp x_3x_2\\
&x_3^2=0, \qquad z_{1}^2=0.
\end{align} 
The dimension of  $\toba(\eny_{\wp}(1))$ is $2^{4}$, since it has a PBW-basis
\begin{align*}
\{ x_{1}^{m_{1}}x_{2}^{m_2} z_{1}^{n_{1}}x_3^{n_0}\,:\, m_i,n_i \in\I_{0,1} \}.
\hspace{50pt} \qed
\end{align*} 
\end{prop}

\begin{prop}\label{nich-pale-2} Let $z_{01}:=\ad_c x_3(z_{1})$. The algebra
 $\toba(\eny_{\wp}(\omega))$ is presented by generators $x_{1},x_2, x_3$ with defining relations
\begin{align}
&x_{1}^2= 0,\qquad  x_{2}^2=0,\qquad x_{1}x_2=x_2x_{1}, \\
x_{1}x_3&=\wp x_3x_{1},\quad z_{1}=x_2x_3+\wp x_3x_2\\
&x_3^3=0, \qquad z_{1}^{3}=0.\\
&z_{01}^3=0, \qquad (\ad_{c} x_3)^{2}(z_{1})=0.
\end{align} 
The dimension of  $\toba(\eny_{\wp}(\omega))$ is $2^{2}3^3$, since it has a PBW-basis
\begin{align*}
\{ x_{1}^{m_{1}}x_{2}^{m_2} z_{1}^{n_2}z_{01}^{n_{1}}x_3^{n_0}\,:\, m_i \in\I_{0,1},\, n_i\in\I_{0,2} \}.
\hspace{50pt} \qed
\end{align*} 
\end{prop}

\subsection{Realizations}\label{subsec:realizations-paleblock-pt}
Assume that $\wp$ is a root of 1 of odd order $M$.
Take $\Gamma = \langle g_{1} \rangle \times \langle g_2 \rangle$ where $g_{1}$ has order $M$ and $g_2$ has order 
$2M$. 
We realize $\eny_{\wp}(1)$ in $\ydG$ by  $\deg x_{1} = \deg x_{2}= g_{1}$, $\deg x_3 = g_2$ and action
\eqref{eq:action-pale}. Then $\toba \big(\eny_{\wp}(1)\big) \# \ku \Gamma$ is a pointed Hopf algebra of dimension $2^5 M^2$.

Also, let $\Upsilon = \langle h_{1} \rangle \times \langle h_2 \rangle$ where 
$h_{1}$ has order $M$ and $h_2$ have order $P := \lcm (6,M)$. 
We realize $\eny_{\wp}(\omega)$ in $\ydU$ by  $\deg x_{1} = \deg x_{2} = h_{1}$, $\deg x_3 = h_2$ and action as in 
\eqref{eq:action-pale} with $h_i$'s instead of the $g_i$'s. 
Then $\toba \big(\eny_{\wp}(\omega)\big) \# \ku \Upsilon$ is a pointed Hopf algebra of dimension $2^3 3^3 MP$.

\end{document}